\documentclass[11pt,reqno,tbtags,a4paper]{amsart}
\usepackage{amssymb}
\usepackage{xpunctuate}
\usepackage{url}
\usepackage[square,numbers]{natbib}
\bibpunct[, ]{[}{]}{;}{n}{,}{,}

\title
{Rounding of discrete variables}

\date{April 8, 2025}

\author{Svante Janson}
\thanks{Supported by the Knut and Alice Wallenberg Foundation
and
the Swedish Research Council
}
\address{Department of Mathematics, Uppsala University, PO Box 480,
SE-751~06 Uppsala, Sweden}
\email{svante.janson@math.uu.se}
\newcommand\urladdrx[1]{{\urladdr{\def~{{\tiny$\sim$}}#1}}}
\urladdrx{http://www2.math.uu.se/~svante/}

\subjclass[2020]{} 

\numberwithin{equation}{section}

\renewcommand\le{\leqslant}
\renewcommand\ge{\geqslant}

\allowdisplaybreaks

\theoremstyle{plain}
\newtheorem{theorem}{Theorem}[section]

\theoremstyle{definition}

\newcommand\xqed[1]{%
    \leavevmode\unskip\penalty9999 \hbox{}\nobreak\hfill
    \quad\hbox{#1}}

\newtheorem{exampleqqq}[theorem]{Example}
\newenvironment{example}{\begin{exampleqqq}}
  {\xqed{$\triangle$}\end{exampleqqq}}

\newtheorem{remarkqqq}[theorem]{Remark}
\newenvironment{remark}{\begin{remarkqqq}}
  {\xqed{$\triangle$}\end{remarkqqq}}

\newtheorem*{ack}{Acknowledgement}

\theoremstyle{remark}

\newcounter{dummy}
\makeatletter
\newcommand\myitem[1][]{\item[#1]\refstepcounter{dummy}\def\@currentlabel{#1}}
\makeatother

\newenvironment{romenumerate}[1][-10pt]{
\addtolength{\leftmargini}{#1}\begin{enumerate}
 \renewcommand{\labelenumi}{\textup{(\roman{enumi})}}%
 }{\end{enumerate}}

\newcounter{oldenumi}
{\setcounter{oldenumi}{\value{enumi}}
\begin{romenumerate} \setcounter{enumi}{\value{oldenumi}}}
{\end{romenumerate}}

\newcounter{thmenumerate}

\newcounter{xenumerate}   

\newcommand{\refT}[1]{Theorem~\ref{#1}}

\newcommand{\refR}[1]{Remark~\ref{#1}}

\newcommand{\refS}[1]{Section~\ref{#1}}

\newcommand{\refSS}[1]{Section~\ref{#1}}

\begingroup
  \divide\count255 by 60
  \multiply\count255 by -60
  \advance\count255 by \time
  \ifnum \count255 < 10 \xdef\klockan{\the\count1.0\the\count255}
  \else\xdef\klockan{\the\count1.\the\count255}\fi
\endgroup

\newcommand{\sumin}{\sum_{i=1}^n}

\newcommand{\sumkn}{\sum_{k=1}^n}
\newcommand{\sumjq}{\sum_{j=0}^{q-1}}

\newcommand{\sumkq}{\sum_{k=0}^{q-1}}
\newcommand{\sumkoooo}{\sum_{k=-\infty}^{\infty}}
\newcommand{\sumjiq}{\sum_{j=1}^{q-1}}

\newcommand{\sumlq}{\sum_{\ell=0}^{q-1}}
\newcommand{\sumlid}{\sum_{\ell=1}^{d-1}}

\newcommand{\prodkn}{\prod_{k=1}^n}

\newcommand\set[1]{\ensuremath{\{#1\}}}

\newcommand\bigpar[1]{\bigl(#1\bigr)}
\newcommand\Bigpar[1]{\Bigl(#1\Bigr)}

\newcommand\bigsqpar[1]{\bigl[#1\bigr]}
\newcommand\sqpar[1]{[#1]}

\newcommand\cpar[1]{\{#1\}}

\newcommand\abs[1]{\lvert#1\rvert}
\newcommand\bigabs[1]{\bigl\lvert#1\bigr\rvert}

\newcommand\lrabs[1]{\left\lvert#1\right\rvert}
\def\rompar(#1){\textup(#1\textup)}    

\def\xexp(#1){e^{#1}}
\newcommand\ceil[1]{\lceil#1\rceil}

\newcommand\floor[1]{\lfloor#1\rfloor}
\newcommand\lrfloor[1]{\left\lfloor#1\right\rfloor}
\newcommand\round[1]{\langle#1\rangle}

\newcommand\punkt{\xperiod}    

\newcommand\eg{e.g\punkt}

\newcommand\ii{\mathrm{i}}

\newcommand\eqd{\overset{\mathrm{d}}{=}}

\newcommand\bbZ{\mathbb Z}

\newcounter{CC}
\newcounter{cc}

\newcommand\E{\operatorname{\mathbb E}{}} 
\renewcommand\P{\operatorname{\mathbb P{}}}

\newcommand\Var{\operatorname{Var}}

\newcommand\gf{\varphi}

\renewcommand\phi{\xxx}  

\newcommand\cE{\mathcal E}

\newcommand\tU{{\widetilde U}}

\newcommand\indic[1]{\boldsymbol1\cpar{#1}}

\newcommand\setoi{\set{0,1}}

\newcommand\ddx{\mathrm{d}}

\newcommand{\chf}{characteristic function}

\newcommand\rhs{right-hand side}

\newcommand\xX{{\floor{X}}}
\newcommand\yX{{\round{X}}}
\newcommand\hh{\widetilde{h}}
\newcommand\del{\mid}
\newcommand\GCD{\operatorname{GCD}}
\renewcommand\GCD{\gcd}

\newcommand{\Holder}{H\"older}

\newcommand\CS{Cauchy--Schwarz}
\newcommand\CSineq{\CS{} inequality}

\begin{document}

\begin{abstract} 
Let $X$ be a random variable that
takes its values in $\frac{1}q\bbZ$, for some integer $q\ge2$,
and consider $X$ rounded to an integer, 
either downwards or upwards or to the nearest integer.
We give general formulas for the characteristic function and moments
of the rounded variable. 
These formulas complement the related but different  formulas in the
case that $X$ has a continuous distribution, which was studied by Janson (2006).
\end{abstract}

\maketitle


\section{Introduction}\label{S:intro}

Let $X$ be a random variable and consider also $X$ rounded to an integer;
more precisely we may consider, for example, $\floor X$ (rounding downwards)
or $\round X$ (rounding to the nearest integer),  
see \refSS{SSnot} for precise definitions.

We gave in \cite{SJ175} general formulas for the characteristic function and
moments (in particular, the first and second moments)
of the rounded variable in the important case when $X$ has a
continuous distribution; this was motivated by several examples that had
appeared as subsequence limits of integer-valued random variables in
different problems.

In the present paper we consider instead the case of $X$ with a discrete
distribution.
More precisely, we suppose that there exists an integer $q\ge2$ such that
$X$ takes its values in $\frac{1}q\bbZ$, or equivalently that $qX\in\bbZ$ a.s.
(For completeness, we allow also $q=1$ below, but the results are trivial in
this case.)
Again, we give general formulas for the characteristic function and moments.

The results can be compared to those given in \cite{SJ175} for the case of a
continuous distribution; the results are similar but different.

\begin{remark}
  Our setting is obviously equivalent to  rounding an integer-valued
random variable to multiples of $q$, and the results can easily be
translated to that case.
\end{remark}

\begin{ack}
  I thank Seungki Kim for inspiring this work by a question, motivated by
  a possible application. It now seems that these result are not needed for 
  that purpose, but I nevertheless collect them here for other possible
  future applications.
\end{ack}

\section{Notation}\label{SSnot}

The characteristic function of a random variable $X$ is denoted by
$\gf_X(t):=\E e^{\ii t X}$.

The indicator function of an event $\cE$ is denoted $\indic{\cE}$.

For two integers $j$ and $k$, $j\del k$ means that $j$ is a divisor of $k$,
i.e., that $k\in j\bbZ$.

For a real number $x$, let $\floor{x}$ and $\ceil{x}$
denote $x$ rounded downwards and upwards, respectively, to the nearest integer.
Similarly,
$\round x$ denotes $x$ rounded to the nearest integer,
for definiteness choosing the larger one if there is a tie.
Thus 
$\floor{x}\le x<\floor{x}+1$,
$\ceil{x}-1< x\le\ceil{x}$
and
$\round{x}-\frac12\le x<\round{x}+\frac12$;
furthermore, 
\begin{align}
\label{ceil}
\ceil{x}&=-\floor{-x},
\\
\label{round}
\round{x}&=\lrfloor{x+\tfrac12}  
.\end{align}

\begin{remark}\label{Rsym}
  Rounding downwards to $\floor X$ is obviously an asymmetric operation;
the mirror operation is to rounding upwards to $\ceil{X}$.
Since $\ceil{X}=-\floor{-X}$ by \eqref{ceil}, 
we will for simplicity consider only $\floor{X}$ and $\round{X}$ below;
results for $\ceil{X}$ follow from the results 
$\floor{X}$ applied to $-X$. In particular, it is easily
seen that 
\eqref{er1} and \eqref{tmom1} below hold for $\ceil{X}$ if we replace
$h_q(t)$ by $h_q(-t)$.

Also rounding to $\round{X}$ is slightly asymmetric since we choose
the larger
value when there is a tie.
The alternative to choose the lower value is again the mirror operation
given by $-\round{-X}$, and results follow from the results below,
now replacing $\hh_q(t)$ by $\hh_q(-t)$.

Note that when $q$ is odd and $qX$ is integer-valued, there cannot be a tie
for the rounding $\round{X}$, so the two versions coincide.
This explains the greater symmetry in some of the results for odd $q$, in
particular that $h_q(t)$ in \eqref{c2o} then is a symmetric function.
\end{remark}

\section{The \chf}\label{Schf}

Let $q\ge1$ be a fixed integer.
Suppose that $X$ is a random variable such that 
$qX$ is integer-valued,
i.e., 
$X\in \frac{1}q\bbZ$ (a.s.).
Note that this assumption implies (and in fact is equivalent to)
$\gf_X$ havíng period $2\pi q$.
We consider the rounded variables $\xX$ and $\yX$.

We first define two auxilliary functions. Let
\begin{align}
  \label{c1}
h_q(t):=\frac{1}{q} \sumkq e^{-\ii t k/q} 
=\frac{1-e^{-\ii t}}{q(1-e^{-\ii t/q})}
=\frac{\sin \frac{t}2}{q\sin\frac{t}{2q}}e^{-\ii\frac{q-1}{2q}t},
\end{align}
where the two last formulas are interpreted by continuity as 1 when the
denominator is 0 (i.e., when $t\in 2\pi q\bbZ$). 
Furthermore, let, if $q$ is even,
\begin{align}
  \label{c2e}
\hh_q(t):=\frac{1}{q} \sum_{k=-q/2}^{q/2-1} e^{-\ii t k/q} 
=e^{\ii t/2}h_q(t)
=\frac{\sin \frac{t}2}{q\sin\frac{t}{2q}}e^{\frac{\ii t}{2q}},
\end{align}
and if $q$ is odd
\begin{align}
  \label{c2o}
\hh_q(t):=\frac{1}{q} \sum_{k=-(q-1)/2}^{(q-1)/2} e^{-\ii t k/q} 
=e^{\ii\frac{q-1}{2q} t}h_q(t)
=\frac{\sin \frac{t}2}{q\sin\frac{t}{2q}}.
\end{align}

Note that these functions may be interpreted as \chf{s}.
Let $U_q$ be a random variable that is uniformly distributed
on $\set{\frac{k}{q}: k=0,\dots,q-1}$, and let 
$\tU_q$ be a random variable that is uniformly distributed on
$\set{\frac{k}{q}: k=-\frac{q}2,\dots,\frac{q}2-1}$ ($q$ even)
or
$\set{\frac{k}{q}: k=-\frac{q-1}2,\dots,\frac{q-1}2}$ ($q$ odd).
Thus, these random variables are uniformly distributed 
on the set of the $q$ integer
multiples of $1/q$ in $[0,1)$ ($U_q$) and $[-\frac12,\frac12)$
($\tU_q$).
Then $h_q(-t)$ is the \chf{} $\gf_{U_q}(t)$ of $U_q$
and $\hh_q(-t)$ is the \chf{} $\gf_{\tU_q}(t)$ of $\tU_q$;
equivalently,
\begin{align}\label{gfU}
h_q(t)=\gf_{-U_q}(t),
\qquad
\hh_q(t)=\gf_{-\tU_q}(t).  
\end{align}
 
\begin{theorem}\label{T1}
Let $q\ge1$ be an integer, and
suppose that $X$ is a random variable such that 
$qX$ is integer-valued.
Then
\begin{align}
  \label{er1}
\gf_\xX(t) =\sumjq  h_q(t+2\pi j)\gf_X(t+2\pi j)
\end{align}
and
\begin{align}
  \label{er2}
\gf_\yX(t) =\sumjq  \hh_q(t+2\pi j)\gf_X(t+2\pi j).
\end{align}
\end{theorem}

\begin{proof}
It is convenient to use the random variables $U_q$ and $\tU_q$ above,
assuming as we may that they are independent of $X$.
Then, for any integer $k$,
\begin{align}\label{er3}
  \P(\floor{X}=k)&
=\sumjq \P\Bigpar{X=k+\frac jq}
=q\sumjq \P\Bigpar{X=k+\frac jq}\P\Bigpar{U_q=\frac jq}
\notag\\&
=q\P\bigpar{X-U_q=k}.
\end{align}
Furthermore, if $Y$ is any random variable with values in $\frac1q\bbZ$, then
\begin{align}\label{er4}
  \sumjq \gf_Y(t+2\pi j)&
= \sumjq \sumkoooo \sumlq \P\bigpar{Y=k+\ell/q}e^{\ii(t+2\pi j)(k+\ell/q)}
\notag\\&
=\sumkoooo \sumlq e^{\ii t(k+\ell/q)}\P\bigpar{Y=k+\ell/q}\sumjq e^{2\pi\ii j\ell/q}
\notag\\&
=q\sumkoooo e^{\ii t k}\P\bigpar{Y=k}
\end{align}
since the inner sum over $j$ vanishes when $\ell\neq0$ and is $q$ for $\ell=0$.
(This argument can be regarded as Fourier inversion on the finite group
$\bbZ_q$). 
Taking $Y:=X-U_q$, we obtain by combining \eqref{er3} and \eqref{er4}
and recalling \eqref{gfU},
\begin{align}
\gf_{\floor{X}}(t)&
=q\sumkoooo e^{\ii tk}\P\bigpar{X-U_q=k}
=
\sumjq \gf_{X-U_q}(t+2\pi j)
\notag\\&
=
\sumjq \gf_{-U_q}(t+2\pi j) \gf_X(t+2\pi j)
=
\sumjq h_q(t+2\pi j) \gf_X(t+2\pi j),
\end{align}
which shows \eqref{er1}.

We obtain \eqref{er2} by the same argument, now using $\tU_q$ instead of $U_q$.
\end{proof}

\begin{remark}
  \label{Rq}
The sums in \eqref{er1} and \eqref{er2} may be taken over any set of $j$
that has exactly one element in each residue class mod $j$,
since the summands have period $q$ in $j$
(because $h_q$, $\hh_q$, and $\gf_X$ all are $2\pi q$-periodic).
The same applies to \eqref{tmom1} and \eqref{tmom2} below.
\end{remark}

\section{Moments}\label{Smom}
Suppose that for some integer $r\ge1$,
$X$ has a finite $r$th moment, i.e., $\E[|X|^r]<\infty$.
Then  obviously $\floor X$ and $\round X$ also have finite $r$th moments.
\refT{T1} implies the following formulas.

\begin{theorem}\label{Tmom}
Let $q\ge1$ be an integer, and
suppose that $X$ is a random variable such that 
$qX$ is integer-valued.
Then, for every integer $r\ge1$ such that $\E|X|^r<\infty$,
\begin{align}
  \label{tmom1}
\E\bigpar{\xX^r} 
=\ii^{-r}\sumjq \frac{\ddx^r}{\ddx t^r} \bigpar{h_q(t)\gf_X(t)}\bigr|_{t=2\pi j}
\end{align}
and
\begin{align}
  \label{tmom2}
\E\bigpar{\yX^r} 
=\ii^{-r}\sumjq \frac{\ddx^r}{\ddx t^r} \bigpar{\hh_q(t)\gf_X(t)}\bigr|_{t=2\pi j}
.\end{align}
\end{theorem}

\begin{proof}
These follow immediately by
 differentiating \eqref{er1} and \eqref{er2} $r$ times at $t=0$.
\end{proof}

In the following two sections, we derive more explicit formulas for the
first two moments.

\section{Mean}\label{Smean}

\begin{theorem}\label{TM}
Let $q\ge1$ be an integer, and
suppose that $X$ is a random variable such that 
$qX$ is integer-valued.
Suppose also $\E|X|<\infty$.
Then
\begin{align}\label{t2}
\E\xX&
=
\E X-\frac12+\frac{1}{2q}+\sumjiq \frac{1}{q(1-e^{-2\pi\ii j/q})}\gf_X(2\pi j).
\end{align}

If $q$ is even, then
\begin{align}\label{t2e}
\E\yX&
=
\E X+\frac{1}{2q}+\sumjiq \frac{(-1)^j}{q(1-e^{-2\pi\ii j/q})}\gf_X(2\pi j).
\end{align}
If $q$ is odd, then
\begin{align}\label{t2o}
\E\yX&
=
\E X+\sumjiq \frac{(-1)^j}{q(e^{\pi\ii j/q}-e^{-\pi\ii j/q})}\gf_X(2\pi j).
\end{align}
\end{theorem}

\begin{proof}
By
taking $r=1$ in \eqref{tmom1},
or directly by taking the derivative of \eqref{er1} at $t=0$,  
we obtain
\begin{align}\label{m1}
  \ii\E\xX&
=
\gf_\xX'(0) =\sumjq  h_q(2\pi j)\gf'_X(2\pi j)+
\sumjq  h_q'(2\pi j)\gf_X(2\pi j)
.\end{align}
We have $\gf_X'(0)=\ii\E X$, 
and \eqref{c1} yields
\begin{align}
  \label{m10}
h_q(0)=1
\quad\text{and}\quad
h_q(2\pi j)=0 \text{ for }1\le j\le q-1.
\end{align}
Hence, \eqref{m1} simplifies to
\begin{align}\label{m11}
  \ii\E\xX&
=
\ii\E X+\sumjq  h_q'(2\pi j)\gf_X(2\pi j)
.\end{align}

The first expression in \eqref{c1} yields
\begin{align}\label{m2}
  h_q'(0)=\frac{1}{q}\sumkq \frac{-\ii k}{q}
=-\ii \frac{q(q-1)/2}{q^2}
=-\ii\frac{q-1}{2q}.
\end{align}
(Alternatively, by \eqref{gfU} this follows from, and is equivalent to, 
$\E U_q=\frac{q-1}{2q}$, which easily is seen directly.) 
Furthermore, 
the second expression in \eqref{c1} yields, noting that the numerator
there vanishes at $t=2\pi j$,
\begin{align}\label{m3}
  h_q'(2\pi j)
=\frac{\ii}{q(1-e^{-2\pi\ii j/q})},
\qquad 1\le j\le q-1.
\end{align}
Consequently, \eqref{m11} yields \eqref{t2}.

For $\E\round{X}$, we argue in the same way using \eqref{er2} or \eqref{tmom2}.
We note that \eqref{c2e}--\eqref{c2o} and \eqref{m10} yield
\begin{align}
  \label{w10}
\hh_q(0)=1
\quad\text{and}\quad
\hh_q(2\pi j)=0 \text{ for }1\le j\le q-1.
\end{align}
Hence, taking derivatives in \eqref{er2} yields
\begin{align}\label{m5}
  \ii\E\yX&
=
\ii\E X+\sumjq  \hh_q'(2\pi j)\gf_X(2\pi j)
.\end{align}

Furthermore, \eqref{c2e} and \eqref{c2o} 
together with \eqref{m10} and \eqref{m2}--\eqref{m3}
imply that:\\
If $q$ is even, then
\begin{align}\label{w2e}
  \hh_q'(0)=h_q'(0)+\frac{\ii}{2}h_q(0)=\frac{\ii}{2q}
\end{align}
and
\begin{align}\label{w3e}
  \hh_q'(2\pi j)
=e^{\ii\pi j} h_q'(2\pi j)
=\ii\frac{(-1)^j}{q(1-e^{-2\pi\ii j/q})},
\qquad 1\le j\le q-1;
\end{align}
if $q$ is odd, then
\begin{align}\label{w2o}
  \hh_q'(0)=h_q'(0)+\ii\frac{q-1}{2q}h_q(0)=0
\end{align}
(which also follows from \eqref{gfU} since $\E \tU_q=0$ when $q$ is odd),
and
\begin{align}\label{w3o}
  \hh_q'(2\pi j)
=e^{\ii\frac{q-1}{q}\pi j } h_q'(2\pi j)
=\ii\frac{(-1)^j}{q(e^{\pi\ii j/q}-e^{-\pi\ii j/q})}
=\frac{(-1)^j}{2q \sin(\pi j/q)}
,\qquad 1\le j\le q-1.
\end{align}
We obtain \eqref{t2e} and \eqref{t2o} by substituting
\eqref{w2e}--\eqref{w3o} into \eqref{m5}.
\end{proof}

\section{Second moment}\label{SMM}

\begin{theorem}\label{TMM}
Let $q\ge1$ be an integer, and
suppose that $X$ is a random variable such that 
$qX$ is integer-valued.
Suppose also $\E[X^2]<\infty$.
Then
\begin{align}\label{mm1}
\E\bigsqpar{\xX^2}&
=\E[X^2]
+\frac{2q^2-3q+1}{6q^2}
-\frac{q-1}{q}\E X
-2\sumjiq \frac{\ii}{q(1-e^{-2\pi\ii j/q})}\gf'_X(2\pi j)
\notag\\&\hskip4em
-\sumjiq \Bigpar{\frac{1}{q(1-e^{-2\pi\ii j/q})}
+2\frac{e^{-2\pi\ii j/q}}{q^2(1-e^{-2\pi\ii j/q})^2}}
\gf_X(2\pi j)
.\end{align}
If $q$ is even, then
\begin{align}\label{mm2e}
\E\bigsqpar{\yX^2}&
=\E[X^2] +\frac{1}{12}+\frac1{6q^2}
+\frac{1}{q}\E X
-2\sumjiq  \ii\frac{(-1)^j}{q (1-e^{2\pi\ii j/q})}\gf'_X(2\pi j)
\notag\\&\hskip4em
+\sumjiq  \frac{(-1)^{j}}{2q^2\sin^2(\pi j/q)}\gf_X(2\pi j)
.\end{align}
If $q$ is odd, then
\begin{align}\label{mm2o}
\E\bigsqpar{\yX^2}&
=\E[X^2] +\frac{1}{12}-\frac1{12q^2}
-\sumjiq  \frac{(-1)^j}{q \sin(\pi j/q)}\gf'_X(2\pi j)
\notag\\&\hskip4em
+\sumjiq  \frac{(-1)^{j}\cos(\pi j/q)}{2q^2\sin^2(\pi j/q)}\gf_X(2\pi j)
.\end{align}
\end{theorem}

\begin{proof}
By taking $r=2$ in \eqref{tmom1}, or directly
by taking the second derivative of \eqref{er1} at $t=0$, 
we obtain,
using $\gf_X''(0)=-\E[X^2]$ and \eqref{m10},
\begin{align}\label{e1}
\E(\xX^2)&
=
-\gf_\xX''(0) 
\notag\\&
=-\sumjq  h_q(2\pi j)\gf''_X(2\pi j)
-2\sumjq  h_q'(2\pi j)\gf'_X(2\pi j)
-\sumjq  h_q''(2\pi j)\gf_X(2\pi j)
\notag\\&
=\E[X^2]
-2\sumjq  h_q'(2\pi j)\gf'_X(2\pi j)
-\sumjq  h_q''(2\pi j)\gf_X(2\pi j)
.\end{align}

We obtain from \eqref{c1}
\begin{align}\label{e2}
  h_q''(0)
= \frac{1}{q} \sumkq \frac{-k^2}{q^2}
=-\frac{q(q-1)(2q-1)}{6q^3}
=-\frac{2q^2-3q+1}{6q^2}
\end{align}
and, since $e^{2\pi j\ii}=1$,
\begin{align}\label{e3}
h_q''(2\pi j)
=\frac{1}{q(1-e^{-2\pi\ii j/q})}
+2
\frac{e^{-2\pi\ii j/q}}{q^2(1-e^{-2\pi\ii j/q})^2},
\qquad j=1,\dots,q-1.
\end{align}
Consequently, using also \eqref{m2}--\eqref{m3}, $\gf_X'(0)=\ii\E X$ and
$\gf_X(0)=1$, 
\eqref{mm1} follows from 
\eqref{e1}.

Turning to $\round X$, we obtain as above, from \eqref{er2} or
\eqref{tmom2},
\begin{align}\label{em1}
&\E(\yX^2)
=\E[X^2]
-2\sumjq  \hh_q'(2\pi j)\gf'_X(2\pi j)
-\sumjq  \hh_q''(2\pi j)\gf_X(2\pi j)
.\end{align}

For even $q$,
we obtain from \eqref{c2e}, \eqref{e2}--\eqref{e3}, 
\eqref{m2}--\eqref{m3}, and \eqref{m10},
\begin{align}\label{em2}
  \hh_q''(0)=h_q''(0)+\ii h_q'(0)-\tfrac14 h_q(0)
=-\frac{q^2+2}{12q^2}
\end{align}
and 
\begin{align}\label{em3}
\hh_q''(2\pi j)&
=(-1)^jh_q''(2\pi j)+2\frac{\ii}{2}(-1)^jh_q'(2\pi j)
\notag\\&
=2\frac{(-1)^je^{-2\pi\ii j/q}}{q^2(1-e^{-2\pi\ii j/q})^2}
=\frac{(-1)^{j+1}}{2q^2\sin^2(\pi j/q)}
,\qquad j=1,\dots,q-1.
\end{align}

For odd $j$, we obtain similarly 
(or directly from the last formula in \eqref{c2o})
\begin{align}\label{em2o}
  \hh_q''(0)
=-\frac{q^2-1}{12q^2}
\end{align}
and 
\begin{align}\label{em3o}
\hh_q''(2\pi j)&
=\frac{(-1)^{j+1}\cos(\pi j/q)}{2q^2\sin^2(\pi j/q)}
,\qquad j=1,\dots,q-1.
\end{align}
The results \eqref{mm2e} and \eqref{mm2o} now follow from \eqref{em1},
using \eqref{w2e}--\eqref{w3o} and \eqref{em2}--\eqref{em3o}.
\end{proof}

\section{Examples}\label{Sex}

\begin{example}\label{E2}
  Consider the simplest case $q=2$, i.e., assume that $X$ a.s.\ is an integer
  or half-integer.
Then $\round X = \ceil X$.
For the mean, we obtain from \refT{TM}
\begin{align}\label{jw1a}
  \E\floor{X} &= \E X -\frac{1}{4} + \frac{\gf_X(2\pi)}{4},
\\\label{jw1b}
  \E\ceil X =\E\round{X} &= \E X +\frac{1}{4} - \frac{\gf_X(2\pi)}{4}.
\end{align}
(This can also easily be seen directly.)
Similarly, 
for the second moment,  from  \refT{TMM},
\begin{align}\label{jw2a}
 \E\bigsqpar{\xX^2}&=\E\sqpar{X^2}+\frac{1}{8}-\frac12\E X 
-\frac{\ii}2\gf_X'(2\pi)-\frac{1}{8}\gf_X(2\pi),
\\\label{jw2b}
\E\bigsqpar{{\ceil X}^2}=
\E\bigsqpar{\yX^2}&
=\E\sqpar{X^2}+\frac{1}{8}+\frac12\E X +\frac{\ii}2\gf_X'(2\pi)
-\frac{1}{8}\gf_X(2\pi).
\end{align}
Note that \eqref{jw1a}--\eqref{jw1b} and \eqref{jw2a}--\eqref{jw2b}
agree with the relation
$\ceil{X}=-\floor{-X}$.
\end{example}

\begin{example}\label{E0}
  Let $q\ge1$, and let $X:=U_q$. This eaxmple is trivial, since obviously
  $\floor{U_q}=0$; nevertheless it is interesting to see how this is reflected
  in the formulas above. Recall that by \eqref{gfU}, 
we now have
  \begin{align}\label{jw3}
    \gf_X(t)=\gf_{U_q}(t)=h_q(-t)=\overline{h_q(t)}.
  \end{align}
In particular, \eqref{m10} shows that
\begin{align}\label{jw4}
  \gf_{U_q}(2\pi j)=0\quad\text{for } 1\le j\le q-1.
\end{align}
Hence the sum in \eqref{t2} vanishes, and \eqref{t2} reduces to
\begin{align}\label{jw5}
  0=\E\floor{U_q}=\E U_q -\frac12+\frac1{2q}
=\E U_q -\frac{q-1}{2q},
\end{align}
or 
\begin{align}\label{jw5b}
\E U_q = \frac{q-1}{2q},   
\end{align}
as is easily seen 
(and was observed after \eqref{m2}).

The variance formula is more interesting.
The last sum in \eqref{mm1} vanishes, again because of \eqref{jw4}.
For the first sum we have, by \eqref{m3} and \eqref{jw3},
\begin{align}\label{jw6}
  \sumjiq \frac{\ii}{q(1-e^{-2\pi\ii j/q})}\gf'_{U_q}(2\pi j)&
=\sumjiq h_q'(2\pi j)\overline{h_q'(2\pi j)}
=\sumjiq \bigabs{h_q'(2\pi j)}^2
\notag\\&
=\sumjiq \frac{1}{\abs{2q\sin(\pi j/q)}^2}
=\frac{1}{4q^2}\sumjiq \frac{1}{\sin^2(\pi j/q)}
.\end{align}
Hence \eqref{mm1} yields, using also \eqref{jw5b}, 
\begin{align}\label{jw7}
  0=\E\bigsqpar{{\floor{U_q}}^2}
=\E[U_q^2] +\frac{2q^2-3q+1}{6q^2}-\frac{(q-1)^2}{2q^2}
-\frac{1}{2q^2}\sumjiq \frac{1}{\sin^2(\pi j/q)}.
\end{align}
Furthermore, by \eqref{jw3} and \eqref{e2},
\begin{align}
  \E[U_q^2]=-\gf_{U_q}''(0)=-h_q''(0)=\frac{2q^2-3q+1}{6q^2},
\end{align}
which equals the second term on the \rhs{} of \eqref{jw7}.
(See the proof of \refT{TMM}.)
Hence \eqref{jw7} is equivalent to
\begin{align}
  \label{jw8}
\sumjiq \frac{1}{\sin^2(\pi j/q)}
=\frac{2}3\bigpar{2q^2-3q+1}-(q-1)^2
=\frac{q^2-1}3.
\end{align}
This non-obvious formula can also be shown more directly by 
applying Parseval's formula to  the restriction of $h_q'$ to 
\set{2\pi  j:0\le j<q}, regarded as a function
on the cyclic group $\bbZ_q$; this function is given
by \eqref{m2}--\eqref{m3}, and it follows from \eqref{c1} that its  inverse
Fourier transform (also a function on $\bbZ_q$) is given by
$-\ii\frac{k}{q}$, $0\le k<q$; calculations similar to (and related to) the
ones above then yield \eqref{jw8}.

We obtain the same result \eqref{jw8} by similar calculations if we instead
take  $X:=\tU_q$ and consider $\round{\tU_q}=0$ in \refT{TMM}.
\end{example}

\begin{example}
  As a more complicated example, we consider the following problem;
we thank Seungki Kim for asking it, which was the original motivation for
the present paper. 

Let $q$ be odd, let $n\ge1$, and let $\xi_1,\dots,\xi_n$ be i.i.d.\  
with each $\xi_k$ having the distribution of $\tU_q$ in \refS{Schf},
i.e., $\xi_k$ is uniformly distributed on 
$\set{-\frac{q-1}{2q},-\frac{q-3}{2q},\dots,\frac{q-3}{2q},\frac{q-1}{2q}}$;
in other words  each $q\xi_k$ is uniformly distributed on the set
$\set{-\frac{q-1}2,-\frac{q-3}2,\dots,\frac{q-3}2,\frac{q-1}2}$
of integers in $(-\frac{q}2,\frac{q}2)$.
Let $s_1,\dots, s_n$ be positive integers, and 
consider
\begin{align}\label{s1}
X := s_1 \xi_1+\dots+ s_n \xi_n.
\end{align}
Since $q$ is odd we have $\E \xi_i=\E\tU_q=0$ and,
using \eqref{gfU} and \eqref{em2o},
\begin{align}\label{ew1}
  \Var\xi_i=\E[\xi_i^2]=\E\bigsqpar{\tU_q^2}=-\hh_q''(0)=\frac{q^2-1}{12q^2}.
\end{align}
Consequently,
\begin{align}\label{ew2}
  \Var X = \sumin s_i^2\Var\xi_i = \frac{q^2-1}{12q^2}\sumin s_i^2.
\end{align}
Consider now the rounded variable $\round{X}$, and in particular its
variance.
A natural approximation to $\Var[\round X]$ is $\Var X+\frac{1}{12}$.
(For any $X$. In statistics this is known as Sheppard's correction 
(for the variance) when dealing with grouped data,
see \eg{} \cite[Section 27.9]{Cramer}.
See also \cite{SJ175} for continuous $X$.)
How good is this approximation for the variable $X$ in \eqref{s1}?

We will in \eqref{ss7} below
give an upper bound of the error that is valid for all $n\ge2$ and
$s_1,\dots,s_n$; 
it is presumably not sharp but it seems to be rather
good when the $s_i$ are much smaller
than $q$. We leave it to the readers to find better bounds,
in particular for cases with larger $s_i$.

Each $\xi_k$ has the \chf{}
\begin{align}\label{s2}
  \gf_\xi(t) 
=\frac{1}q \sum_{j=-(q-1)/2}^{(q-1)/2}e^{\ii jt/q}
=\hh_q(t)
=\frac{\sin \frac{t}2}{q\sin\frac{t}{2q}},
\end{align}
see \eqref{c2o} and \eqref{gfU}.
Consequently, $X$ has the \chf{}
\begin{align}\label{s3}
  \gf_X(t)=\prodkn \gf_\xi(s_kt) 
=\prodkn \hh_q(s_kt)
.\end{align}

It follows from \eqref{w10} and the fact that $\hh_q$ has period $2\pi q$
that for every $j\in\bbZ$, we have
$\gf_X(2\pi j)\in\setoi$, and, letting $\GCD(\dots)$ denote the greatest
common divisor,
\begin{align}\label{s4}
  \gf_X(2\pi j)\neq0 &
\iff  \gf_X(2\pi j)=1
\iff \gf_\xi(2\pi js_k)=1\;\forall k \le n
\iff q\del js_k\;\forall k \le n
\notag\\&
\iff q\del j\GCD(s_1,\dots,s_n,q).
\end{align}
Let 
\begin{align}
d&:=  \GCD(s_1,\dots,s_n,q),\label{s5}
\\
J&:=q/d.\label{s6}
\end{align}
Note that $d$ and $J$ are divisors of $q$, and thus odd positive integers with
$d,J\le q$. Furthermore, \eqref{s4} can be written
\begin{align}\label{s7}
  \gf_X(2\pi j) = \indic{q\del jd} = \indic{J\del j}.
\end{align}

Each $\xi_k$ has a distribution that is symmetric:
$-\xi_k\eqd\xi_k$. Thus the same holds for $X$, and thus also for $\round X$
(since $q$ is odd, cf.\ \refR{Rsym}).
In particular, 
\begin{align}\label{sa1}
\E\round X=0.  
\end{align}

For the second moment, we use \eqref{mm2o}.
Consider first the last sum in \eqref{mm2o}.
By \eqref{s7}, we may sum over $j=\ell J$, $\ell=1,\dots,d-1$, only,
and thus the absolute value of the sum is 
\begin{align}\label{sq1}
\lrabs{\sumjiq  \frac{(-1)^{j}\cos(\pi j/q)}{2q^2\sin^2(\pi j/q)}\gf_X(2\pi j)}
\le\sumlid \frac{1}{2q^2\sin^2(\pi \ell J/q)}
= \frac{1}{2q^2}\sumlid \frac{1}{\sin^2(\pi \ell /d)}.
\end{align}
The final sum is $(d^2-1)/3$ by \eqref{jw8}.
Hence, \eqref{sq1} implies
\begin{align}\label{sq1a}
\lrabs{\sumjiq  \frac{(-1)^{j}\cos(\pi j/q)}{2q^2\sin^2(\pi j/q)}\gf_X(2\pi j)}
< \frac{d^2}{6q^2}
\le\frac{(\min_k s_k)^2}{6q^2}.
\end{align}

For the first sum in \eqref{mm2o}, we first 
define
\begin{align}\label{sq2}
  X_k:=\sum_{i\neq k} s_i\xi_i,
\end{align}
the sum \eqref{s1} with the term $s_k\xi_k$ omitted.
Then we 
differentiate \eqref{s3}
and obtain
\begin{align}\label{sq3}
  \gf_X'(t)=\sumkn s_k \gf'_\xi(s_kt)\prod_{i\neq k} \gf_\xi(s_it) 
=
\sumkn s_k \gf'_\xi(s_kt)\gf_{X_k}(t)
.\end{align}
Let
\begin{align}
d_k&:=  \GCD(\set{s_i:i\neq k}\cup\set q),\label{sq5}
\\
J_k&:=q/d_k.\label{sq6}
\end{align}
Then \eqref{s7} applied to $X_k$ yields, for all $j\in\bbZ$,
\begin{align}\label{sq7}
  \gf_{X_k}(2\pi j) = \indic{q\del jd_k} = \indic{J_k\del j}.
\end{align}
Consequently, 
for the first sum in \eqref{mm2o},
using \eqref{sq3}, \eqref{s2}, and \eqref{w2o}--\eqref{w3o}, 
and taking $j=\ell J_k$
similarly to the argument in \eqref{sq1},
\begin{align}\label{sq8}
  \sumjiq \frac{(-1)^j}{q \sin(\pi j/q)}\gf'_X(2\pi j)
&= \sumjiq  \frac{(-1)^j}{q \sin(\pi j/q)}
  \sumkn s_k\gf'_\xi(2\pi s_kj)\gf_{X_k}(2\pi j)
\notag\\&
=\sumkn s_k \sumjiq  \frac{(-1)^j}{q \sin(\pi j/q)}\cdot 
\frac{(-1)^{s_kj}\indic{q\nmid s_kj}}{2q \sin(\pi s_kj/q)}\indic{J_k\del j}
\notag\\&
=\sumkn \frac{s_k}{2q^2} \sum_{\ell=1}^{d_k-1}  
\frac{(-1)^{(s_k+1)\ell}\indic{d_k\nmid s_k\ell}}
 {\sin(\pi \ell/d_k)\sin(\pi s_k\ell/d_k)}.
\end{align}
Denote the inner sum on the last line by $S_k$.
Then the \CSineq{} yields
\begin{align}\label{sq9}
  S_k^2 \le 
\sum_{\ell=1}^{d_k-1} \frac{1}{\sin^2(\pi \ell/d_k)}
\cdot
\sum_{\ell=1}^{d_k-1}\frac{\indic{d_k\nmid s_k\ell}}{\sin^2(\pi s_k\ell/d_k)}.  
\end{align}
The first sum is $<d_k^2/3$ by \eqref{jw8}. For the second sum,
we note that by \eqref{s5} and \eqref{sq5}, 
\begin{align}\label{ss1}
d=\gcd(d_k,s_k)  
\end{align}
and thus
\begin{align}\label{ss2}
  d_k\del s_k\ell \iff d_k \del \gcd(s_k\ell,d_k\ell)=d\ell
\iff (d_k/d) \del \ell.
\end{align}
Let $r_k:=d_k/d$ (an odd integer by \eqref{ss1} and \eqref{sq5})
and $s_k':=s_k/d$ (also an integer by \eqref{ss1}).
It follows that $\gcd(s_k',r_k)=1$; thus $s_k'\ell$ runs through the
equivalence classes modulo ${r_k}$ once each as $\ell=1,\dots,r_k$,
and $d_k/r_k=d$ times each as $\ell=1,\dots,d_k$.
Consequently, since $s_k\ell/d_k=s'_k\ell/r_k$, using \eqref{jw8} again,
\begin{align}\label{ss3}
  \sum_{\ell=1}^{d_k-1}\frac{\indic{d_k\nmid s_k\ell}}{\sin^2(\pi s_k\ell/d_k)}
= 
d \sum_{j=1}^{r_k-1}\frac{1}{\sin^2(\pi j/r_k)}
<\frac{d r_k^2}{3} = \frac{d_k^2}{3d}
\le \frac{d_k^2}{3}.
\end{align}
Hence, \eqref{sq9} yields $S_k^2\le (d_k^2/3)^2$ and thus $|S_k|\le d_k^2/3$,
and  \eqref{sq8} yields
\begin{align}\label{ss4}
\lrabs{\sumjiq \frac{(-1)^j}{q \sin(\pi j/q)}\gf'_X(2\pi j)}  
\le\frac{1}{6q^2} \sumkn s_k d_k^2
.\end{align}

Assume now $n\ge2$. Then $d_k\le s_{k+1}$ (with the index taken modulo $n$),
and thus \Holder's inequality (with exponents $3$ and $3/2$) shows
\begin{align}\label{ss5}
  \sumkn s_kd_k^2 
\le \Bigpar{\sumkn s_k^3}^{1/3}\Bigpar{\sumkn d_k^3}^{2/3}
\le \sumkn s_k^3.
\end{align}
Consequently, 
\eqref{mm2o} yields, using \eqref{ss4}--\eqref{ss5} and \eqref{sq1a},
\begin{align}\label{ss7}
\lrabs{ \E\bigsqpar{\yX^2}-\bigpar{\E[X^2] +\tfrac{1}{12}}}
\le \frac{1+\sumkn s_k^3+(\min_k s_k)^2}{6q^2}
\le \frac{\sumkn s_k^3}{3q^2}.
\end{align}
\end{example}

\newcommand\AAP{\emph{Adv. Appl. Probab.} }
\newcommand\JAP{\emph{J. Appl. Probab.} }
\newcommand\JAMS{\emph{J. \AMS} }
\newcommand\MAMS{\emph{Memoirs \AMS} }
\newcommand\PAMS{\emph{Proc. \AMS} }
\newcommand\TAMS{\emph{Trans. \AMS} }
\newcommand\AnnMS{\emph{Ann. Math. Statist.} }
\newcommand\AnnPr{\emph{Ann. Probab.} }
\newcommand\CPC{\emph{Combin. Probab. Comput.} }
\newcommand\JMAA{\emph{J. Math. Anal. Appl.} }
\newcommand\RSA{\emph{Random Structures Algorithms} }
\newcommand\DMTCS{\jour{Discr. Math. Theor. Comput. Sci.} }

\newcommand\AMS{Amer. Math. Soc.}
\newcommand\Springer{Springer-Verlag}
\newcommand\Wiley{Wiley}

\newcommand\vol{\textbf}
\newcommand\jour{\emph}
\newcommand\book{\emph}
\newcommand\inbook{\emph}
\def\no#1#2,{\unskip#2, no. #1,} 
\newcommand\toappear{\unskip, to appear}

\newcommand\arxiv[1]{\texttt{arXiv}:#1}
\newcommand\arXiv{\arxiv}

\newcommand\xand{and }
\renewcommand\xand{\& }

\def\nobibitem#1\par{}

\end{document}